\documentclass[a4paper, 12pt]{amsart}

\usepackage{amsmath, amsthm, amsfonts, amssymb}
\usepackage[hidelinks]{hyperref}
\usepackage[english]{babel}
\usepackage{mathrsfs}
\usepackage{eucal}
\usepackage[all]{xy}
\usepackage{tikz}
\usepackage{verbatim}
\allowdisplaybreaks

\usepackage{amssymb}

\newtheorem{thm}{Theorem}[section]
\newtheorem*{thm*}{Theorem}
\newtheorem{lem}[thm]{Lemma}
\newtheorem*{prob*}{Problem}

\newtheorem{prop}[thm]{Proposition}
\newtheorem*{prop*}{Proposition}

\newtheorem{cor}[thm]{Corollary}
\newtheorem*{cor*}{Corollary}

\theoremstyle{definition}
\newtheorem{defn}[thm]{Definition}
\newtheorem*{defn*}{Definition}
\newtheorem{notation}[thm]{Notation}
\newtheorem{remark}[thm]{Remark}

\newtheorem{example}[thm]{Example}

\newtheorem*{question*}{Question}
\newtheorem*{Pquestion*}{Popa's question}

\newtheorem*{conv*}{Convention}

\def\cal{\mathcal}

\makeatletter

\def\dotminussym#1#2{%
  \setbox0=\hbox{$\m@th#1-$}%
  \kern.5\wd0%
  \hbox to 0pt{\hss\hbox{$\m@th#1-$}\hss}%
  \raise.6\ht0\hbox to 0pt{\hss$\m@th#1.$\hss}%
  \kern.5\wd0}

\newcommand\ee{{\varepsilon}}

\newcommand{\inp}[2]{\left\langle#1,#2\right\rangle}
\newcommand{\twoone}{II$_1$ }

\newcommand{\cB}{ {\cal B} }

\newcommand{\cF}{ {\cal F} }

\newcommand{\cU}{ {\cal U} }
\newcommand{\cC}{ {\cal C} }

\newcommand{\cH}{ {\cal H} }
\newcommand{\cK}{ {\cal K} }

\newcommand{\cN}{ {\cal N} }

\newcommand{\cR}{ {\cal R} }

\newcommand{\bN}{ {\mathbb{N}} }

\newcommand{\bR}{ {\mathbb{R}} }
\newcommand{\bC}{ {\mathbb{C}} }
\newcommand{\bZ}{ {\mathbb{Z}} }

\newcommand{\keys}[1]{\small \textbf{\textit{Keywords---}} #1}

\begin{document}

\title{Primeness of Generalized Wreath Product \twoone Factors}

\author{Gregory Patchell}
\email{gpatchel@ucsd.edu}
\address{Department of Mathematical Sciences\\University of California San Diego}
\maketitle

\begin{abstract}
In this article we investigate the primeness of generalized wreath product \twoone factors using deformation/rigidity theory techniques. We give general conditions relating tensor decompositions of generalized wreath products to stabilizers of the associated group action and use this to find new examples of prime \twoone factors.
\end{abstract}

\keys{Von Neumann Algebras, II$_1$ Factors, Functional Analysis, Operator Algebras, Primeness, Generalized Bernoulli Actions, Deformation/Rigidity}

\section{Introduction}
\label{sect-intro}

Tensor products of von Neumann algebras have been considered for as long as von Neumann algebras have, harkening all the way back to Murray and von Neumann's seminal work \cite{murray1936rings}. However, it was nearly a half-century before an explicit example of a \textit{prime} \twoone factor was given by Popa \cite{popa1983orthogonal}; that is, he showed that the free group factors with uncountably many generators admit no nontrivial tensor decompositions. Fifteen years later, Ge showed that all free group factors are prime \cite{ge1998applications}, and in 2003 Ozawa showed that $L\Gamma$ is prime for all icc hyperbolic groups $\Gamma$ \cite{ozawa2004some}. In 2006, Peterson additionally showed that $L\Gamma$ is prime for all icc groups $\Gamma$ with positive first $\ell^2$-Betti number \cite{peterson20092}.

Besides von Neumann algebras arising from groups, another natural class of algebras comes from group actions. In 2006, using his deformation/rigidity theory, Popa showed a primeness result related to the Bernoulli action of a nonamenable group: if $B$ is a nontrivial amenable tracial von Neumann algebra and $\Gamma$ is nonamenable, then $B^\Gamma \rtimes \Gamma$ is prime \cite{popa2008superrigidity}. More recently, in 2019, Isono and Marrakchi showed the same result but where $B$ is allowed to be \textit{any} nontrivial von Neumann algebra \cite{isono2019tensor}.

Both Popa and Isono and Marrakchi exploit the mixingness of the Bernoulli action \cite{popa2008superrigidity,isono2019tensor}. Indeed, the action $\Gamma\curvearrowright \Gamma$ has finite (in fact trivial) stabilizers, so $B^\Gamma\rtimes\Gamma$ is mixing relative to $L\Gamma$ (see, e.g., \cite{boutonnet2014several}). Isono and Marrakchi also exploit the rigidity of \textit{full factors}, namely \twoone factors with no nontrivial central sequences. (I.e., non-$\Gamma$ factors. See \cite{connes1976classification}.) They combine rigidity with the recently discovered property of \textit{weak bicentralization} \cite{bannon2020full} to make deductions about certain categories of bimodules. 

In this paper, we use the framework of Popa's deformation/rigidity theory and apply Isono and Marrakchi's techniques to a more general type of actions, namely, generalized Bernoulli actions, which have been studied in, for example, \cite{ioana2013class,vaes2008explicit,popa2008strong}. This immediately opens up many new possibilities. In particular, the action may no longer be mixing and could have very large stabilizers. We will later see that this greatly impacts the categories of bimodules. We are still able to deduce primeness for a large class of examples:

\begin{thm}
\label{cor-main}
Let $\Gamma\curvearrowright I$ be a faithful, nonamenable, transitive action of a non-inner amenable countable group on a countable set such that $\mathrm{Stab}(F)$ is amenable for some finite set $F\subset I.$ Let $(B,\tau)$ be a \twoone factor. Then $M = B^I\rtimes\Gamma$ is prime.
\end{thm}

Our proof uses crucially the fact that $M$ is a full factor. Our most general conditions to ensure fullness are a nonamenable action with the VV property (defined in Section \ref{subsect-vv}). In particular, a nonamenable action of a non-inner amenable group suffices. 

We note that nonamenable, transitive actions are of the form $\Gamma\curvearrowright\Gamma/\Lambda$ where $\Lambda$ is not co-amenable in $\Gamma.$ A faithful action of $\Gamma$ on $\Gamma/\Lambda$ corresponds to $\Lambda$ having trivial normal core. We can also see that if $\Lambda$ is weakly malnormal in $\Gamma$, then there is some $F\subset \Gamma/\Lambda$ such that $\mathrm{Stab}(F)$ is finite, and in particular, amenable. Therefore, we have the following immediate corollary to Theorem \ref{cor-main}:

\begin{cor}
\label{cor-main-gp1}
    Let $\Gamma$ be a non-inner amenable group and $\Lambda < \Gamma$ a weakly malnormal subgroup that is not co-amenable with trivial normal core. Let $(B,\tau)$ be a \twoone factor. Then $M = B^{\Gamma/\Lambda}\rtimes\Gamma$ is prime. 
\end{cor}

In the case that $\Lambda$ is amenable, we can generalize Theorem C of Isono and Marrakchi \cite{isono2019tensor}:

\begin{thm}
\label{cor-main-2}
    Let $\Gamma\curvearrowright I$ be a free, nonamenable action with the VV property of a countable group on a countable set such that $\mathrm{Stab}(i)$ is amenable for all $i\in I$. Let $(B,\tau)$ be a nontrivial tracial von Neumann algebra. Then $M = B^I\rtimes\Gamma$ is prime.
\end{thm}

Theorem \ref{cor-main} and Theorem \ref{cor-main-2} follow from our more general but somewhat technical theorem:

\begin{thm}
\label{thm-main}
Let $\Gamma\curvearrowright I$ be a faithful, nonamenable action with the VV property of a countable group on a countable set. Let $(B,\tau)$ be a diffuse tracial von Neumann algebra. Write $M = B^I\rtimes\Gamma$ and suppose $M = P\overline{\otimes} Q$ is a tensor decomposition. Then at least one of the following occurs:

\begin{enumerate}
    \item[(a)] $P\prec_M L(\mathrm{Stab}(F))$ for all finite sets $F\subset I$.
    \item[(b)] $Q \prec_M (Z(B)^F\overline{\otimes} B^{I\setminus F})\rtimes\mathrm{Stab}(F)$ for some finite set $F\subset I.$ \qedhere
\end{enumerate} 
\end{thm}

To date, the only examples the author knows of where condition (a) occurs with a tensor factor in a generalized Bernoulli crossed product \twoone factor is when $P$ is hyperfinite, such as in Example \ref{eg-supernonVV} and Example \ref{eg-non-VV}. It would be interesting to know whether such examples exist where $P$ is not hyperfinite (or moreover when $P$ is full). In the event condition (a) cannot occur, then condition (b) must be true for both $P$ and $Q$ simultaneously. This situation can arise, for example, when $\Gamma\curvearrowright I$ decomposes as in Example \ref{eg-two-orbits}. 

Section \ref{sect-prelim} will define the relevant terms needed for the proof. Section \ref{sect-examples} will highlight some pertinent examples of generalized Bernoulli crossed products. Section \ref{sect-rigidfactor} will prove the main result under the added assumption that one of the tensor factors is rigidly embedded in the crossed product. Section \ref{sect-nonrigid-factor} will contain the rest of proof of the main result. Section \ref{sect-corollaries} will give applications.

\subsection*{Declaration of Interests} This work was supported in part by the National Science Foundation [FRG \#1854074].

\subsection*{Acknowledgements} I would like to thank my advisor, Adrian Ioana, for suggesting this project and for the innumerable hours of discussions, guidance, and mentorship. I would also like to thank Hui Tan for her helpful comments and insights. An anonymous referee also provided several useful comments.

\section{Preliminaries}
\label{sect-prelim}
\subsection{Averaging, conditional expectations, and commuting squares}
\label{subsect-cond-exp}
Tracial von Neumann algebras $(B,\tau)$ are \textit{finite}, i.e., every isometry in $B$ is a unitary. One of our most powerful tools when working with finite objects is averaging, and the setting of finite von Neumann algebras is no different. Our first lemma is technical but guarantees that certain 2-norm limits remain in $B$ instead of living in $L^2(B)\setminus B.$ 

\begin{lem}
\label{lem-technical-avging}
    Let $(B,\tau)$ be a tracial von Neumann algebra and let $K\subset B$ be a $\|\cdot\|_\infty$-bounded convex set. Then the $\|\cdot\|_2$-closure of $K$ in $B$ has a unique element of minimal 2-norm.
\end{lem}

\begin{proof}
    The $\|\cdot\|_2$-closure of $K$ is contained in $B$ since $K$ is $\|\cdot\|_\infty$-bounded and the unit ball in $B$ is $\|\cdot\|_2$-complete (see Proposition 2.6.4 of \cite{anantharaman2017introduction}). Since $K$ is a compact convex subset of the Hilbert space $L^2(B),$ it has a unique minimal element $v.$
\end{proof}

Furthermore, if $A\subset (B,\tau)$ is an inclusion of tracial von Neumann algebras, i.e., such that $\tau_A = (\tau_B)|_A$, then there is a unique trace-preserving \textit{conditional expectation} $E_A: B\to A$ (see Theorem 9.1.2 of \cite{anantharaman2017introduction}). In many cases, there is a way to use averaging to compute this conditional expectation.

\begin{lem}
\label{lem-expect-onto-commutant}
    Let $A\subset (B,\tau)$ be an inclusion of tracial von Neumann algebras. Fix $x\in B$ and set $$K := \overline{\mathrm{conv}\{uxu^*: u\in \cU(A)\}}^{\|\cdot\|_2}.$$ Then every $z\in K$ such that $uzu^*=z$ for all $u\in\cU(A)$ satisfies $z = E_{A'\cap B}(x)$. In particular, if $y$ is the minimal 2-norm element of $K$ then $E_{A'\cap B}(x) = y.$
\end{lem}

\begin{proof}
     Let $z\in K$ such that $uzu^*=z$ for all $u\in\cU(A)$. Since every element of $A$ is a linear combination of unitaries, $z\in A'\cap B.$ We note that for all $u\in A,$ $E_{A'\cap B}(uxu^*) = E_{A'\cap B}(x)$ which, via convex combinations and 2-norm limits, implies that $E_{A'\cap B}(x) = E_{A'\cap B}(z) = z.$

    By Lemma \ref{lem-technical-avging}, let $y\in B$ be the minimal 2-norm element of $K$. Then for all unitaries $u\in \cU(A)$, $uyu^* \in K$ and $\|uyu^*\|_2 = \|y\|_2$ so $uyu^* = y$. Hence $y\in A'\cap B.$ Therefore, $y = E_{A'\cap B}(x)$. 
\end{proof}

Let $A,B\subset (M,\tau)$ be subalgebras of a tracial von Neumann algebra $M$. Sometimes, the trace-preserving conditional expectations $E_A$ and $E_B$ commute. In this case, we say that $A$ and $B$ form a commuting square and it follows that $E_A\circ E_B = E_{A\cap B}$. We record two simple lemmas related to commuting squares:

\begin{lem}
\label{lem-commutant-is-commuting-square}
    If $A\subset B \subset (M,\tau)$, then $A'\cap M$ and $B$ are subalgebras that form a commuting square.
\end{lem}

\begin{proof}
    Write $C = A'\cap M.$ Then for $x\in M,$ $E_C(x)$ is equal to the element of minimal 2-norm in the closed convex hull of $\{uxu^*: u\in \cU(A)\}$. Note that $E_B(uxu^*) = uE_B(x)u^*$, and since $E_B$ is 2-norm continuous and linear, it follows that $E_C(E_B(x)) = E_B(E_C(x)).$
\end{proof}

\begin{lem}
\label{lem-comm-sq-orthog}
    Let $A,B\subset (M,\tau)$ form a commuting square. Assume further that $A \not\subset B$. Then there exists $0\neq a\in A$ such that $E_B(a) = 0.$ 
\end{lem}

\begin{proof}
    Let $y\in A\setminus B.$ Then $E_B(y) = E_B(E_A(y)) = E_A(E_B(y)) \in A.$ Hence $x = y - E_B(y) \in A$, and clearly $E_B(x) = 0.$ $x\neq 0$ since $y\not\in B.$
\end{proof} 

\subsection{Group actions and crossed products}
\label{subsect-crossed-prod}

Throughout, $\Gamma$ is a countable discrete group and $I$ is a countable set. We say that an action $\Gamma\curvearrowright I$ is \textit{faithful} if for any $g\in\Gamma\setminus\{e\},$ there exists $i\in I$ such that $gi\neq i.$ We say that an action $\Gamma\curvearrowright I$ is \textit{free} if for any $g\in\Gamma\setminus\{e\},$ there are infinitely many $i\in I$ such that $gi\neq i.$ We say an action is \textit{nonamenable} if there is no $\Gamma$-invariant mean on $I.$ In particular, this forces $\Gamma\curvearrowright I$ to have infinite orbits. 

For a finite subset $F\subset I,$ we define $\mathrm{Stab}(F) = \{g\in\Gamma: gi=i \text{ for all } i\in F\}$. If $F = \{i\}$ is a singleton, we abuse notation and write simply $\mathrm{Stab}(\{i\}) = \mathrm{Stab}(i)$. We similarly define $\mathrm{Norm}(F) = \{g\in\Gamma: gF= F\}$. We note that $\mathrm{Stab}(F)$ is always a finite index subgroup of $\mathrm{Norm}(F)$.

If $\Gamma\curvearrowright I$ is an action and $B$ is a tracial von Neumann algebra, we define $B^I = \otimes_{i\in I} B$. We note that the infinite tensor product $B^I$ is the 2-norm closed span of the \textit{pure tensors}, i.e., the elements $\otimes_{i\in I}x_i$ where $x_i \in B^i$ for all $i\in I$ and only finitely many of the $x_i$ are not equal to 1. The group $\Gamma$ acts on $B^I$ in the following, trace-preserving way on pure tensors: $g \cdot (\otimes_{i\in I}x_i) = \otimes_{i\in I}x_{g^{-1}i}.$ We denote this action by $\sigma_g,$ and it extends to all of $B^I$ via linearity and $\|\cdot\|_2$-continuity. We may then form the crossed product von Neumann algebra $M = B^I\rtimes\Gamma.$ Note that $M$ has a natural trace $\tau$ that restricts to the traces on $B^I$ and on $L\Gamma$.

We record three lemmas we will need later related to tensor automorphisms, i.e., the maps $\sigma_g\in \mathrm{Aut}(B^I)$. The first one says that when $B$ is diffuse, every nontrivial tensor automorphism is \textit{properly outer}, that is, that if there is $x\in B^I$ such that $x\sigma_g(b) = bx$ for all $b\in B^I,$ then $x=0.$

\begin{lem}
\label{lem-tensor-outer}

Let $A$ be a diffuse tracial von Neumann algebra. Let $I$ be a countable set. Denote by $A^I$ the (possibly infinite) tensor power $\bigotimes_{i\in I} A$. Let $\rho$ be a permutation of the set $I$. There is a corresponding automorphism $\sigma$ of $A^I$ such that $\sigma(\otimes_{i\in I}a_i) = \otimes_{i\in I}a_{\rho^{-1}(i)}$ for every pure tensor $\otimes_{i\in I}a_i\in A^I.$ If $\rho$ is a nontrivial permutation of $I,$ then $\sigma$ is a properly outer automorphism of $A^I$. 
\end{lem}

\begin{proof}
    Take $(u_n)_n$ to be a sequence of unitaries in $A$ weakly converging to 0. Pick $i\in I$ such that $\rho(i) = j\neq i.$ Then set $y_n^i = u_n,$ $y_n^k = 1$ for $k\neq i,$ and $y_n = \bigotimes_{i\in I} y_n^i$ (a pure tensor with $u_n$ in the $i$th coordinate and 1s elsewhere). 
    
    We claim that for all $a\in A^I,$ $\inp{ay_n}{\sigma(y_n)a} \to 0.$ If the claim holds, then clearly we cannot have $ay_n = \sigma(y_n)a$ for any nonzero $a\in A^I$ and so the action $\sigma$ would be properly outer. 

    To prove the claim, we fix $x = \bigotimes_{k\in I}x_k, z=\bigotimes_{k\in I}z_k$ to be pure tensors, where only finitely many of the coordinates are not equal to 1. Then $xy_n = x_iu_n \bigotimes_{k\neq i}x_k$ and $\sigma(y_n)z = u_nz_j \bigotimes_{k\neq j}z_k$. Hence $$\tau(xy_nz^*\sigma(y_n)^*) = \tau(x_iu_nz_i^*)\tau(x_jz_j^*u_n^*)\prod_{k\neq i,j}\tau(x_kz_k^*) \to 0$$ as $n\to \infty$ since the $u_n$ go weakly to 0. Taking linear combinations and 2-norm limits of pure tensors proves the claim for all $a\in A^I.$
\end{proof}

We obtain a similar result if $A$ is merely assumed nontrivial but the action is free:

\begin{lem}
\label{lem-tensor-outer-free}

Let $A$ be a nontrivial tracial von Neumann algebra. Let $I$ be a countably infinite set. Denote by $A^I$ the infinite tensor power $\bigotimes_{i\in I} A$. Let $\rho$ and $\sigma$ be as in the previous lemma. If $\rho$ is a permutation such that $\rho(i)\neq i$ for infinitely many $i\in I,$ then $\sigma$ is a properly outer automorphism of $A^I$. 
\end{lem}

\begin{proof}
    Pick $u\in A\ominus \bC1$. Pick $i_1,i_2,\ldots\in I$ to be distinct points such that $\rho(i_n) = j_n\neq i_n$ for all $n\geq 1$. Then set $y_n^{i_n} = u,$ $y_n^k = 1$ for $k\neq i_n,$ and $y_n = \bigotimes_{i\in I} y_n^i$ (a pure tensor with $u$ in the $i_n$th coordinate and 1s elsewhere). 
    
    We claim that for all $a\in A^I,$ $\inp{ay_n}{\sigma(y_n)a} \to 0.$ If the claim holds, then clearly we cannot have $ay_n = \sigma(y_n)a$ for any nonzero $a\in A^I$ and so the action $\sigma$ would be properly outer. 

    To prove the claim, we fix $x = \bigotimes_{k\in I}x_k, z=\bigotimes_{k\in I}z_k$ to be pure tensors, where only finitely many of the coordinates are not equal to 1. Then $xy_n = x_{i_n}u\bigotimes_{k\neq i_n}x_k$ (that is, $xy_n$ is a pure tensor with $x_{i_n}u$ in the $i_n^{th}$ component and $x_k$ in all of the other components). Similarly, $\sigma(y_n)z = uz_{j_n}\bigotimes_{k\neq j_n}z_k$. Hence $$\tau(xy_nz^*\sigma(y_n)^*) = \tau(x_{i_n}uz_{i_n}^*)\tau(x_{j_n}z_{j_n}^*u^*)\prod_{k\neq i_n,j_n}\tau(x_kz_k^*).$$ As $n\to \infty$, $i_n$ escapes the supports of $x$ and $z$ and so $\tau(x_{i_n}uz_{i_n}^*) = 0$ for all $n$ sufficiently large. Hence $\tau(xy_nz^*\sigma(y_n)^*)\to0.$ Taking linear combinations and 2-norm limits of pure tensors proves the claim for all $a\in A^I.$
\end{proof}

As a consequence of Lemmas \ref{lem-tensor-outer} and \ref{lem-tensor-outer-free} and Proposition 5.2.3 of \cite{anantharaman2017introduction}, we have that $M = B^I\rtimes \Gamma$ is a \twoone factor whenever (i) $B$ is a nontrivial, tracial von Neumann algebra and $\Gamma\curvearrowright I$ is a free action with infinite orbits, or (ii) $B$ is a diffuse tracial von Neumann algebra and $\Gamma\curvearrowright I$ is a faithful action with infinite orbits.

The third lemma we record here allows us to commute certain commutants inside of crossed product von Neumann algebras.

\begin{lem}
\label{lem-outer-commutant}
    Let $A$ be a tracial von Neumann algebra and let $G$ be a finite group acting on $A$ by properly outer automorphisms $\sigma_g$. Denote by $A^G$ the elements of $A$ fixed by the action of $G.$ Then $(A^G)'\cap A = Z(A).$
\end{lem}

\begin{proof}
    It is clear that $Z(A)\subset (A^G)'\cap A.$ We prove the other inclusion. Let $x\in (A^G)'\cap A$ be a unitary. Then for all $y\in A,$ $\sum_{g\in G}\sigma_g(y) \in A^G,$ so $\sum_{g\in G}\sigma_g(y)x = \sum_{g\in G}x\sigma_g(y).$ Multiplying both sides by $\sum_{h\in G}\sigma_h(y^*)x^*$ and taking the trace gives the equation
    $$\sum_{g,h\in G} \tau(\sigma_g(y)x\sigma_h(y)^*x^*) = \sum_{g,h\in G} \tau(\sigma_g(y)\sigma_h(y)^*).$$
    Now consider the set
    $$K = \overline{\mathrm{conv}}^{\|\cdot\|_2}\{\bigoplus_{g,h\in G}\sigma_g(y)x\sigma_h(y)^*\oplus\bigoplus_{g,h\in G}\sigma_g(y)\sigma_h(y)^*:y\in \cU(A)\}\subset A^{\oplus 2|G|^2}$$
    By Lemma \ref{lem-technical-avging}, $K$ has a unique element of minimal 2-norm $c\in A^{\oplus 2|G|^2}$. Write $c = ((a_{g,h})_{g,h\in G},(b_{g,h})_{g,h\in G}).$ Since $c$ is unique, we must have that 
    $$((a_{g,h})_{g,h\in G},(b_{g,h})_{g,h\in G}) = ((\sigma_g(y)a_{g,h}\sigma_h(y)^*)_{g,h\in G},(\sigma_g(y)b_{g,h}\sigma_h(y)^*)_{g,h\in G})$$
    for all $y\in \cU(A).$ In other words, for each $g,h\in G,$ and $y\in\cU(A)$, $\sigma_g(y)a_{g,h} = a_{g,h}\sigma_h(y)$ and $\sigma_g(y)b_{g,h} = b_{g,h}\sigma_h(y)$. Since the action is properly outer, this implies that $a_{g,h}=b_{g,h} = 0$ for all $g\neq h.$ It is also clear that $a_{g,g}\in Z(A)$ and $b_{g,g} = 1$ for all $g\in G.$ In fact, $a_{g,g} = E_{Z(A)}(x)$ by Lemma \ref{lem-expect-onto-commutant}.

    Let $\ee>0.$ We can now choose, $t_i\in \bR$ and $y_i\in\cU(A)$ such that for all $g\neq h$ in $G,$ the following four inequalities hold:
    \begin{align*}
        \|\sum_{i=1}^n t_i\sigma_g(y_i)x\sigma_h(y_i)^*\|_2 &< \ee &
        \|\sum_{i=1}^n t_i\sigma_g(y_i)\sigma_h(y_i)^*\|_2 &< \ee \\
        \|\sum_{i=1}^n t_i\sigma_g(y_i)x\sigma_g(y_i)^* - E_{Z(A)}(x)\|_2 &< \ee &
        \|\sum_{i=1}^n t_i\sigma_g(y_i)\sigma_g(y_i)^* - 1\|_2 &< \ee \\
    \end{align*}

    Since $x$ is a unitary, we further deduce that for all $g\in G$ that
    $$\|\sum_{i=1}^n t_i\sigma_g(y_i)x\sigma_g(y_i)^*x^* - E_{Z(A)}(x)x^*\|_2 < \ee.$$
    Now, 
    \begin{align*}
        ||G|\tau(E_{Z(A)}(x)x^*) - |G|| &= |\sum_{g,h\in G}\delta_{g,h}(\tau(E_{Z(A)}(x)x^*)-1)| \\
        &\leq |\sum_{g,h} \tau(\delta_{g,h}E_{Z(A)}(x)x^* -\sum_{i=1}^n t_i\sigma_g(y_i)x\sigma_h(y_i)^*x^*)| \\
        &+ |\sum_{g,h} \tau(\sum_{i=1}^n t_i\sigma_g(y_i)x\sigma_h(y_i)^*x^* - \sum_{i=1}^n t_i\sigma_g(y_i)\sigma_h(y_i)^*)| \\
        &+ |\sum_{g,h} \tau(\sum_{i=1}^n t_i\sigma_g(y_i)\sigma_h(y_i)^* - \delta_{g,h})| \\
        &< 2|G|^2\ee,
    \end{align*}
    since the middle term is 0 and $|\tau(a)| \leq \|a\|_2.$ Therefore $|\tau(E_{Z(A)}(x)x^*) - 1| < 2|G|\ee.$ Since this holds for all $\ee>0,$ we see that $\tau(E_{Z(A)}(x)x^*) = 1 = \tau(xx^*).$ But this can only happen if $E_{Z(A)}(x) = x,$ so $x\in Z(A)$.
\end{proof}

\subsection{Ultrapowers, central sequences, full factors, and the VV property}
\label{subsect-vv}
Let $(M,\tau)$ be a tracial von Neumann algebra, and let $\omega$ denote a free ultrafilter on $\bN$. Then $M^\omega$ denotes the following algebra:
$$M^\omega := \{(x_n)_{n\in\bN}\colon x_n\in M, \ \sup_{n\in\bN}\|x_n\|<\infty\}/I_\omega, $$

where $I_\omega$ is the following ideal:

$$I_\omega := \{(x_n)_{n\in\bN}\colon x_n\in M, \ \sup_{n\in\bN}\|x_n\|<\infty, \ \lim_{n\to\omega}\|x_n\|_2 = 0\}.$$

The algebra $M$ embeds diagonally into $M^\omega$ via the map $x\mapsto (x_n)_n$. We note that $M^\omega$ is tracial too, with trace $\tau^\omega((x_n)) = \lim_{n\to\omega}\tau(x_n).$

We define a \textit{central sequence} of $M$ to be a $\|\cdot\|_\infty$-bounded sequence $(x_n)$ of elements in $M$ such that $\|x_ny-yx_n\|_2\to0$ for all $y\in M$. We say a $\|\cdot\|_\infty$-bounded sequence $(x_n)$ of elements in $M$ is \textit{asymptotically trivial} if $\|x_n-\tau(x_n)1\|_2\to 0.$ 

We can now define a \textit{full factor} as a \twoone factor $M$ such that $M'\cap M^\omega = \bC1,$ where we view $M$ as a subalgebra of $M^\omega$ via the diagonal embedding. Equivalently, a \twoone factor $M$ is full if and only if every central sequence of $M$ is asymptotically trivial.

In order to use the machinery of Isono and Marrakchi, we require $M = B^I\rtimes\Gamma$ to be a full factor. When $\Gamma\curvearrowright I$ is nonamenable, Vaes and Verraedt showed that the fullness of $M$ is equivalent to the following condition (see \cite{vaes2015classification}, Lemma 2.7):
\begin{itemize}
    \item Every $\|\cdot\|_\infty$-bounded central sequence $(y_n)_n$ in $ L\Gamma$ satisfying \newline $\|y_n-E_{L(\mathrm{Stab}(i))}(y_n)\|_2\to0$ for all $i\in I$ is asymptotically trivial.
\end{itemize}
We call this condition the \textit{VV property}. We note that $\Gamma\curvearrowright I$ having the VV property is equivalent to the algebra $(L\Gamma)' \cap \bigcap_{i\in I}(L(\mathrm{Stab}(i)))^\omega$ being trivial.

\begin{prop}
\label{prop-vv-ez}
An action $\Gamma\curvearrowright I$ has the VV property if either of the following is satisfied:
\begin{enumerate}
    \item[(a)] $\Gamma $ is not inner amenable.
    \item[(b)] There exists $F\subset I$ finite such that $\mathrm{Stab}(F) = \{e\}.$ \qedhere
\end{enumerate}
\end{prop}

\begin{proof} We include short proofs for clarity:
    \begin{enumerate}
        \item[(a)] If $\Gamma$ is not inner amenable, then $L\Gamma$ does not have Property $\Gamma$ (Proposition 15.4.2 of \cite{anantharaman2017introduction}). Therefore all central sequences in $L\Gamma$ are already trivial.
        \item[(b)] If $\mathrm{Stab}(F) = \{e\},$ then $L(\mathrm{Stab}(F)) = \bC1,$ and therefore $E_{L(\mathrm{Stab}(F))} = \tau.$ We observe that the conditional expectations $E_{L(\mathrm{Stab}(i))}$ all commute. Let $(x_n)_n$ be a $\|\cdot\|_\infty$-bounded central sequence in $L\Gamma.$ Then $\|x_n-\tau(x_n)1\|_2 = \|x_n - E_{L(\mathrm{Stab}(F))}(x_n)\|_2\to 0$ if and only if for all $i\in F$ we have that $\|x_n - E_{L(\mathrm{Stab}(i))}(x_n)\|_2\to 0$. \qedhere
    \end{enumerate}
\end{proof}

\subsection{Intertwining by bimodules}
\label{subsect-popa-thm}
The reader is directed to \cite{anantharaman2017introduction}, and in particular Chapters 9, 13, and 17, for more information about the Jones basic construction, bimodules, and intertwining by bimodules respectively.

We recall the \textit{Jones construction} for a tracial von Neumann algebra $(M,\tau)$ and a von Neumann subalgebra $B\subset M$: $L^2(B)$ naturally embeds as a subspace of $L^2(M),$ and there is a projection $e_B \in \cB(L^2(M))$ with range $L^2(B).$ Then we denote by $\inp{M}{B}$ the algebra $(M\cup\{e_B\})'' \subset \cB(L^2(M)).$ The Jones construction $\inp{M}{B}$ has a canonical normal, faithful, semi-finite trace $\hat{\tau}$.

We recall too the notion of a bimodule $_M\cH_N$, namely, it is a *-representation of $M\otimes N^{op}$ on $\cB(\cH)$ whose restrictions to $M$ and $N^{op}$ respectively are normal.

We now state Popa's Theorem (appearing, for example, as Theorem 2.1 of \cite{popa2006strong}). 

\begin{thm}
\label{thm-popa-fundamental}
Let $P,Q$ be von Neumann subalgebras of a tracial von Neumann algebra $(M,\tau).$ Then the following are equivalent:

\begin{enumerate}
    \item[(a)] There is no net $(u_i)$ of unitary elements in $P$ such that for every $x,y\in M,$ $\lim_i \|E_Q(x^*u_iy)\|_2=0$.
    \item[(b)] There exists a nonzero $h\in \inp{M}{Q}_+ \cap P'$ with $\hat{\tau}(h)<\infty$.
    \item[(c)] There exists a nonzero $P$-$Q$-subbimodule $\cH$ of $L^2(M)$ such that $\mathrm{dim}(\cH_Q)<\infty$.
    \item[(d)] There exists an integer $n\geq 1$, a projection $q\in M_n(Q)$, a nonzero partial isometry $v\in M_{1,n}(M)$ and a normal unital homomorphism $\theta\colon P\to qM_n(Q)q$ such that $v^*v\leq q$ and $xv = v\theta(x)$ for all $x\in P$.
    \item[(e)] There exist nonzero projections $p\in P$ and $q\in Q$, a unital normal homomorphism $\theta\colon pPp\to qQq$ and a nonzero partial isometry $v\in pMq$ such that $xv=v\theta(x)$ for all $x\in pPp.$ Moreover, $vv^* \in (pPp)'\cap pMp$ and $v^*v\in \theta(pPp)'\cap qMq.$ \qedhere
\end{enumerate}
\end{thm}

When any of the above equivalent conditions are satisfied, we write $P\prec_M Q$ and say that $P$ embeds into a corner of $Q$ in $M$.

We now recall the notion of $P$-$Q$-finiteness. Given von Neumann subalgebras $P,Q \subset M,$ we say that $x\in  M$ is $P$-$Q$-finite if there exist $x_1,\ldots,x_n,y_1,\ldots,y_m\in M$ such that $xQ\subset \sum_{i=1}^n Px_i$ and $Px\subset \sum_{j=1}^m y_jQ.$ 

By Popa's Theorem, if $M$ has a nonzero $P$-$Q$ finite element then $P\prec_M Q$ and $Q\prec_M P.$ (Take, for example, the closure of $PxQ.$ This is a nonzero $P$-$Q$-subbimodule of $L^2(M)$ with finite right $Q$-dimension.)

We take as the definition of $\mathrm{QN}_{M}(Q)$, the quasi-normalizer of $Q$ in $M,$ to be the set of $Q$-$Q$-finite elements of $M.$

Given two $M$-$N$-bimodules $\cH$ and $\cK,$ we have two notions of containment:
\begin{enumerate}
    \item[(a)] We say $\cH$ is contained in $\cK$ and write $\cH\subset\cK$ if there is an $M$-$N$-bimodular isometry $V:\cH\to\cK.$
    \item[(b)] We say $\cH$ is weakly contained in $\cK$ and write $\cH\prec \cK$ if $\|\pi_\cH(x)\|\leq \|\pi_\cK(x)\|$ for all $x\in M\otimes N^{op}$ (where $\pi_\cH$ is the representation of $M\otimes N^{op}$ in $\cB(\cH)$ and similarly for $\pi_\cK$). 
\end{enumerate}

We note that if $L^2(P) \subset L^2(\inp{M}{Q})$ as $P$-$P$-bimodules then $P\prec_M Q.$

\subsection{Weak containment and relative amenability}
\label{subsect-weakcont}

Related to weak containment is the notion of \emph{relative amenability} of subalgebras. The following theorem and definition are due to Ozawa and Popa (Theorem 2.1 of \cite{ozawa2010class}, Theorem 13.4.4 of \cite{anantharaman2017introduction}):

\begin{thm}
\label{thm-rel-amen}
    Let $Q,N\subset (M,\tau)$ be tracial von Neumann algebras. The following are equivalent:
    \begin{enumerate}
        \item[(a)] There exists an $N$-central state $\varphi$ on $\inp{M}{Q}$ such that $\varphi|_M = \tau$.
        \item[(b)] There exists an $N$-central state $\varphi$ on $\inp{M}{Q}$ such that $\varphi$ is normal on $M$ and faithful on $Z(N'\cap M)$.
        \item[(c)] There exists a conditional expectation $\Phi$ from $\inp{M}{Q}$ onto $N$ such that $\Phi|_M = E_N.$
        \item[(d)] There exists a net $(\xi)_n$ in $L^2(\inp{M}{Q})$ such that $\lim_{n\to\infty}\inp{x\xi_n}{\xi_n} = \tau(x)$ for every $x\in M$ and $\lim_{n\to\infty} \|[y,\xi_n]\|_2 = 0$ for every $y\in N.$
        \item[(e)] $_ML^2(M)_N \prec _ML^2(\inp{M}{Q})_N$.\qedhere
    \end{enumerate}
\end{thm}

\begin{defn}
    Let $Q,N\subset (M,\tau)$ be tracial von Neumann algebras. If any of the conditions of Theorem \ref{thm-rel-amen} hold, we say that $N$ is \emph{amenable relative to} $Q$ \emph{inside} $M$, and write $N \lessdot_M Q.$
\end{defn}

The following basic fact appears in slightly different forms in Proposition 3.1 of \cite{anantharaman1979action} and Theorem 3.2.4(3) of \cite{popa1986correspondences}. We reproduce the proof sketch from \cite{popa1986correspondences} for completeness.

\begin{lem}
\label{lem-rel-amen-crossed-prod}
    Let $\Gamma$ be an amenable group, $(B,\tau)$ a tracial von Neumann algebra, and $\Gamma\curvearrowright B$ a trace-preserving action. Then $B\rtimes\Gamma \lessdot_{B\rtimes\Gamma} B$. In particular, we have that $L^2(B\rtimes\Gamma)\prec L^2(B\rtimes\Gamma)\otimes_{B}L^2(B\rtimes\Gamma)$.
\end{lem}

\begin{proof}
    Let $F_n$ be a sequence of Følner sets for $\Gamma$ that increase to $\Gamma.$ The vectors $\xi_n = \frac{1}{\sqrt{|F_n|}} \sum_{g\in K_n}u_ge_Bu_g^*$ satisfy the conditions of Theorem \ref{thm-rel-amen}(d).

     By Theorem \ref{thm-rel-amen}(e), we know that $_{B\rtimes\Gamma}L^2(B\rtimes\Gamma)_{B\rtimes\Gamma} \prec _{B\rtimes\Gamma}L^2(\inp{B\rtimes\Gamma}{B})_{B\rtimes\Gamma}$.
\end{proof}

\begin{lem}
\label{lem-rel-amen-bimods}
    Let $\Gamma$ be an amenable group, $(B,\tau)$ a tracial von Neumann algebra, and $\Gamma\curvearrowright B$ a trace-preserving action. Set $M = B\rtimes\Gamma$. Then $ L^2(\inp{M}{B\rtimes\Lambda})\prec L^2(\inp{M}{B})$
\end{lem}

\begin{proof}
    By Lemma \ref{lem-rel-amen-crossed-prod}, we have that as $B\rtimes\Lambda$-$B\rtimes\Lambda$-bimodules, $$L^2(B\rtimes\Lambda)\prec L^2(\inp{B\rtimes\Lambda}{B}) = L^2(B\rtimes\Lambda) \otimes_B L^2(B\rtimes\Lambda).$$ Therefore as $M$-$M$-bimodules we have
    \begin{align*}
        L^2(\inp{M}{B\rtimes\Lambda}) &= L^2(M)\otimes_{B\rtimes\Lambda}L^2(M)\\
        &=L^2(M) \otimes_{B\rtimes\Lambda} L^2(B\rtimes\Lambda) \otimes_{B\rtimes\Lambda} L^2(M) \\
        &\prec L^2(M) \otimes_{B\rtimes\Lambda} L^2(B\rtimes\Lambda) \otimes_B L^2(B\rtimes\Lambda) \otimes_{B\rtimes\Lambda} L^2(M) \\
        &= L^2(M) \otimes_B L^2(M)\\
        &= L^2(\inp{M}{B}). 
    \end{align*} \qedhere
\end{proof}

Finally, we would like to record an important characterization of weak containment in the non-tracial setting. We will need it in the semi-finite setting when working with the Jones construction. The following Theorem appears as Corollary A.2 in \cite{bannon2020full}.

\begin{thm}
\label{thm-weak-contain-bmo}
    Let $N\subset M$ be von Neumann algebras. There is a (not necessarily normal) conditional expectation from $M$ onto $N$ if and only if $_NL^2(N)_N\prec _NL^2(M)_N.$
\end{thm}

\section{Examples of Generalized Bernoulli Crossed Products}
\label{sect-examples}

In this section we provide a handful of examples of generalized Bernoulli crossed products in order to highlight the peculiarities that can arise. In particular, there are free, nonamenable actions that give rise to non-prime \twoone factors, are very far from having the VV property, or have very poorly behaved stabilizer subgroups.

\begin{example}
\label{eg-two-orbits}
Let $\Gamma_i\curvearrowright I_i$, $i=\{1,2\}$, be free actions with infinite orbits. Then $\Gamma := \Gamma_1\times \Gamma_2 \curvearrowright I:= I_1 \sqcup I_2$ defined by $(g_1,g_2)\cdot j := g_i j$ where $j\in I_i$ is also a free action with infinite orbits. Then if $B$ is any von Neumann algebra, $M = B^I\rtimes \Gamma = (B^{I_1}\rtimes \Gamma_1) \overline{\otimes} (B^{I_2}\rtimes \Gamma_2).$ It is therefore certainly far too much to hope that every nonamenable action gives rise to a prime \twoone factor.
\end{example}

\begin{example}
\label{eg-supernonVV}
Let $\Gamma_0$ be a nonamenable group and $\Gamma = \bigoplus_{n\in\bN}\Gamma_n$ where $\Gamma_n\simeq \Gamma_0$ for all $n.$ Set $I_n = \Gamma_n$ as sets and let $I = \sqcup_{i\in\bN}I_n $. There is an action of $\Gamma$ on $I$ where $(g_1,g_2,\ldots) \cdot i = g_ki $ for all $i\in I_k.$ $\Gamma$ then acts on $I$ amenably since any sequence $(i_n)_n$ with $i_n\in I_n$, interpreted as a sequence in $\ell^2(I),$ is a sequence of almost invariant vectors. Also, for each $i\in I,$ $\mathrm{Stab}(i) = \oplus_{k\neq j}\Gamma_k$ where $i\in I_j$.

Consider a sequence $(x_n)$ in  $L\Gamma$ with the following property: for each $n,$ $x_n$ is in $L\Gamma_n.$ That is, $x_n = 1\otimes\cdots\otimes1\otimes a_n\otimes1\otimes1\otimes\cdots$ for some $a_n\in L\Gamma_n.$ Then for all $i\in I,$ $\lim \|x_n - E_{L(\mathrm{Stab}(i))}(x_n)\|_2\to0,$ as for only one $n$ we have that $L\Gamma_n$ does not lie inside of $L\mathrm{Stab}(i)$. Furthermore, such sequences asymptotically commute with $L\Gamma$: without loss of generality, $y\in L\Gamma = \otimes_{n\in\bN}L\Gamma_n$ is a pure, finitely supported tensor. But then for $n$ sufficiently large, $[x_n,y] = 0.$ 

This example illustrates that the VV property can fail miserably -- here, the algebra $(L\Gamma)' \cap \bigcap_{i\in I}(L(\mathrm{Stab}(i)))^\omega \subset (L\Gamma)^\omega$ contains all such elements $(x_n)\in (L\Gamma)^\omega$ as described above. On the contrary, for an action to be VV it must satisfy $\bC1 = (L\Gamma)' \cap \bigcap_{i\in I}(L(\mathrm{Stab}(i)))^\omega.$ \qedhere
\end{example}

\begin{example}
\label{eg-non-VV}
Let $\Gamma_0$ be a nonamenable group and $\Gamma = \oplus_{n\in\bN}\Gamma_n$ where $\Gamma_n\simeq \Gamma_0$ for all $n.$ Set $I_n = \Gamma_n$ as sets and set $I = \sqcup_{n\in\bN} I_n.$ Define the action of $\Gamma$ on $I$ in such a way that $\Gamma_n$ acts on $I_n$ on the left, and on $I_{n+1}$ on the right. I.e., if $g\in \Gamma_n$ then $g\cdot i = \begin{cases}gi, & i\in I_n \\ ig^{-1}, & i\in I_{n+1}\\ i, & \text{ otherwise}\end{cases}.$ 

We note that $\Gamma\curvearrowright I$ does not decompose into a direct sum of groups acting on a disjoint union of sets as above, so if $B$ is a nontrivial tracial von Neumann algebra, $M = B^I\rtimes\Gamma$ doesn't admit a tensor decomposition in the same way as above; however, it still is not prime. Indeed, $M$ here will be McDuff, as it will have nontrivial (and non-commuting) central sequences coming from $L\mathrm{Stab}(F)$, $F\subset I$ finite. Therefore the hyperfinite factor $R$ will be a tensor factor of $M$. This example again shows the necessity of the VV property, and that there are nontrivial ways for $M$ to be non-prime. \qedhere

\end{example}

\begin{example}

Consider the action $SL_3(\bZ[1/p]) \curvearrowright SL_3(\bZ[1/p])/SL_3(\bZ)$ by left multiplication of cosets. This is a faithful, transitive action with infinite orbits with the property that $\mathrm{Stab}(F) = \cap_{g\in F}gSL_3(\bZ)g^{-1}$ is finite index inside of $\mathrm{Stab(i)}$ for any $i\in F.$ This last property is simply saying that $SL_3(\bZ)$ is a commensurated subgroup of $SL_3(\bZ[1/p])$. $SL_3(\bZ[1/p])$ has Proeprty (T) so in particular it is not inner amenable, so this action also has the VV property.

This example illustrates that there isn't necessarily a distinction between the conditions $P\prec_M L(\mathrm{Stab}(F))$ for some $F\subset I$ finite and $P\prec_M L(\mathrm{Stab}(F))$ for all $F\subset I$ finite. It is unclear whether any tensor factor of $L(SL_3(\bZ[1/p]))$ can embed in $L(\mathrm{Stab(SL_3(\bZ))})$ for some $i\in I$. We note that $L(SL_3(\bZ[1/p]))$ is conjectured to be prime. \qedhere

\end{example}

\section{The Peculiar Case of a Rigid Tensor Factor}
\label{sect-rigidfactor}
In this section, we will prove the following theorem:

\begin{thm}
\label{thm-case1}
Let $\Gamma\curvearrowright I$ be a free action with infinite orbits of a countable group on a countable set. Let $(B,\tau)$ be a nontrivial tracial von Neumann algebra. Write $M = B^I\rtimes\Gamma$ and suppose $M = P\overline{\otimes} Q$ is a tensor decomposition. Set $\Tilde{B} = B * L\bZ$ and $\Tilde{M} = \Tilde{B}^I\rtimes\Gamma$. If $Q'\cap \Tilde{M}^\omega \subset M^\omega,$ then $P\prec_M L(\mathrm{Stab}(F))$ for every $F\subset I$ finite.
\end{thm}

\begin{remark}
    Theorem \ref{thm-case1} also holds if we assume $\Gamma\curvearrowright I$ is merely faithful if we add the assumption that $B$ is diffuse. These conditions ensure that $\Gamma \curvearrowright B^I$ is a properly outer action (see Lemmas \ref{lem-tensor-outer} and \ref{lem-tensor-outer-free}) and thus that $M$ is a \twoone factor.
\end{remark}

We use the following deformation introduced in \cite{ioana2007rigidity}. We reproduce its construction for clarity and completeness. Let $v\in L\bZ$ be the canonical unitary generator and choose $h\in L\bZ$ self-adjoint with $v = e^{ih}$. Denote by $\alpha_t^0\in \mathrm{Aut}(\Tilde{B})$ the inner automorphism $\alpha_t^0 = \mathrm{Ad}(e^{ith})$. Set $\alpha_t = \bigotimes_{i\in I} \alpha_t^0$. Extend $\alpha_t$ to an automorphism of $\Tilde{M}$ by setting $\alpha_t(u_g) = u_g.$ We observe that $\alpha_t\to\mathrm{id}_{\Tilde{M}}$ pointwise in $\|\cdot\|_2$ as $t\to 0$.

We also define $\beta_0 \in \mathrm{Aut}(\Tilde{B})$ by $\beta_0(v)=v^*$ and $\beta_0(b)=b$ for $b\in B.$ $\beta = \otimes_{i\in I} \beta_0$ is then an automorphism of $\Tilde{B}^I.$ $\beta$ extends to an automorphism of $\Tilde{M}$ by setting $\beta(u_g) = u_g.$ We note that $\beta|_M = \mathrm{id}_M$, $\beta^2 = \mathrm{id},$ and $\beta\circ\alpha_t\circ\beta=\alpha_{-t}$. 

The following is a transversal lemma due to Popa (Lemma 2.1 of \cite{popa2008superrigidity}):

\begin{lem}
\label{popa-transv}
    In the context of the previous three paragraphs, for all $x\in M,$ $$\|\alpha_{2t}(x)-x\|_2 \leq 2\|\alpha_t(x) - E_M(\alpha_t(x))\|_2.$$
\end{lem}

We now reproduce Popa's spectral gap argument  \cite{popa2008superrigidity}. We use the notation of the proof of Theorem 2 in \cite{chifan2010ergodic}.

\begin{lem}
\label{spectral-gap}
    $\alpha_t \to \mathrm{id}$ uniformly in $\|\cdot\|_2$ on $(P)_1.$\qedhere
\end{lem}

\begin{proof}
    Let $\ee>0.$ Since $Q'\cap \Tilde{M}^\omega\subset M^\omega,$ there are unitaries $u_1,\ldots,u_n \in Q$ and $\delta>0$ such that if $x\in (\Tilde{M})_1$ satisfies $\|[u_i,x]\|_2 \leq \delta$ for all $1\leq i\leq n$ then $\|x-E_M(x)\|_2 \leq \ee.$ Let $t_0>0$ be such that for all $t\in [0,t_0]$ and all $1\leq i\leq n$ we have $\|\alpha_t(u_i)-u_i\|_2\leq\delta/2$. Therefore, for all $x\in (P)_1$ and $1\leq i\leq n$ we have that
    $$\|[u_i,\alpha_t(x)]\|_2 \leq 2\|\alpha_t(u_i)-u_i\|_2 + \|[\alpha_t(u_i),\alpha_t(x)]\|_2 \leq \delta. $$
    Therefore we have that $\|\alpha_t(x) - E_M(\alpha_t(x))\|_2 \leq \ee.$ Lemma \ref{popa-transv} then implies \newline $\|\alpha_{2t}(x)-x\|_2 \leq 2\ee$ for all $x\in (P)_1$ and $t\in[0,t_0]$, showing that $\alpha_t$ converges uniformly to $\mathrm{id}$ on $(P)_1.$
\end{proof}

In particular, Lemma \ref{spectral-gap} says that there is $n\geq1$ such that $\tau(u^*\alpha_{1/2^n}(u)) \geq 1/2$ for all $u\in \cU(P).$ The argument from Theorem 4.2 of \cite{ioana2013class}, in particular from the second to last paragraph on page 248 to the second to last paragraph on page 249, can then be used to prove that $P\prec_M B^F\rtimes\mathrm{Stab}(F)$ for some $F\subset I$ finite. Note that the assumption that $\mathrm{Stab}(J)$ is finite for $J$ sufficiently large is not used until the argument in the final paragraph of page 249 of \cite{ioana2013class}. We record the results of the discussion of this section up to this point as a theorem:

\begin{thm}
\label{thm-BF-StabF}
    Let $B$ be a nontrivial, tracial von Neumann algebra, $\Gamma\curvearrowright I$ be a free action with infinite orbits, and $M = B^I\rtimes\Gamma$. Set $\Tilde{B} = B * L\bZ$ and $\Tilde{M} = \Tilde{B}^I\rtimes\Gamma$. If $M = P\overline{\otimes}Q$ is a tensor decomposition such that $Q'\cap \Tilde{M}^\omega \subset M^\omega,$ then $P\prec_M B^F\rtimes \mathrm{Stab}(F)$ for some $F\subset I$ finite (and possibly empty).
\end{thm}

Using the fact that $P$ tensor factor of $M,$ we can in fact intertwine $P$ into just $L(\mathrm{Stab}(F))$:

\begin{thm}
\label{thm-exists-StabF}
    Let $B$ be a nontrivial tracial von Neumann algebra and $\Gamma\curvearrowright I$ a free action with infinite orbits. Let $F\subset I$ be a finite subset. If $M = P\overline{\otimes}Q = B^I\rtimes \Gamma$ is a tensor decomposition such that $P\prec_MB^F\rtimes\mathrm{Stab}(F)$, then $P\prec_M L(\mathrm{Stab}(F)).$
\end{thm}

\begin{proof}
Suppose towards a contradiction that $P\not\prec_M L\mathrm{Stab}(F)$. By Remark 3.8 of Vaes \cite{vaes2008explicit} there are projections $p\in P$ and $q\in B^F\rtimes \mathrm{Stab}(F)$ and a *-homomorphism $\theta:pPp \to q(B^F\rtimes\mathrm{Stab}(F)) q$ and a nonzero partial isometry $v\in pMq$ such that $xv=v\theta(x)$ for all $x\in pPp$, $vv^* \in pPp'\cap pMp = pQ, $ $v^*v \in \theta(pPp)'\cap qMq$ and $\theta(pPp) \not\prec_{B^F\rtimes \mathrm{Stab}(F)} L\mathrm{Stab}(F).$ 

Hence there are unitaries $u_n\in \theta(pPp)$ such that $\|E_{L\mathrm{Stab}(F)}(xu_ny)\|_2 \to 0$ for all $x,y\in B^F\rtimes \mathrm{Stab}(F).$ 

We now define $L := \{g\in\Gamma : gF\cap F= \emptyset\}$ and prove the following claim:

\underline{\textbf{Claim A:}} For all $x\in \overline{\mathrm{span}\{bu_g: b\in B^I, \ g\in L\}}^{\|\cdot\|_2}$ and $y\in L^2(M),$  $$\inp{u_nxu_n^*}{y}\to0.$$ 

It suffices to prove the claim on a $\|\cdot\|_2$-densely spanning set, so without loss of generality let $x = (\otimes_{i\in I}x_i)u_k$ where $x_i$ is a unitary for all $i\in I$, $x_i\neq 1$ for finitely many $i\in I,$ and $k\in L.$ Write $u_n = \sum_{g\in\mathrm{Stab}(F)}u_n^gu_g,$ where $u_n^g\in B^F.$ Let $y = (\otimes_{i\in I}y_i)u_\ell$ where $y_i$ is a unitary for all $i\in I$, $y_i\neq1$ for finitely many $i\in I,$ and $\ell\in \Gamma.$ A computation shows that:

\begin{equation}
\label{eqn-fourier}
\begin{split}
    \tau(u_nxu_n^*y^*) &= \sum_{g,h\in\mathrm{Stab}(F)}\tau(u_n^g(\otimes_{i\in I}x_{g^{-1}i})\sigma_{gkh^{-1}}(u_n^h)^*(\otimes_{i\in I}y^*_{\ell hk^{-1}g^{-1}i})u_{gkh^{-1}\ell^{-1}}) \\
    &= \sum_{g\in\mathrm{Stab}(F)\cap \ell\mathrm{Stab}(F)k^{-1}}\tau(u_n^g(\otimes_{i\in I}x_{g^{-1}i})\sigma_{\ell}(u_n^{\ell^{-1}gk})^*(\otimes_{i\in I} y_i^*))
\end{split}
\end{equation}
Equation \ref{eqn-fourier} followed by noting that the trace was only nonzero when $gkh^{-1}\ell^{-1} = e$, so that $\ell = gkh^{-1} \in \mathrm{Stab}(F)\cdot L\cdot\mathrm{Stab}(F) = L.$ Therefore $F$ and $\ell F$ are disjoint for such $\ell$ and we have that the right hand side of Equation \ref{eqn-fourier} is equal to: 

\begin{align*}
    \sum_{g\in\mathrm{Stab}(F)\cap \ell\mathrm{Stab}(F)k^{-1}}&\tau(u_n^g(\otimes_{i\in F}x_{i}y_i^*))\\
    &\tau((\otimes_{i\in \ell F}x_{g^{-1}i})\sigma_\ell(u_n^{\ell^{-1}gk})^*(\otimes_{i\in \ell F}y_i^*))\\
    & \tau((\otimes_{i\in I\setminus(F\cup \ell F)}x_{g^{-1}i}y_i^*))
\end{align*}

But note that since $|\tau(ab)|\leq\|a\|_2\|b\|_2$ and the $x_i$ and $y_i$ are unitaries, we get that for each $g\in\Gamma$ and $i\in I,$ $|\tau((\otimes_{i\in I\setminus(F\cup \ell F)}x_{g^{-1}i}y_i^*))|\leq 1$. Furthermore, $\inp{u_nxu_n^*}{y} =  \tau(u_nxu_n^*y^*)$ and hence 
\begin{align*}
|\inp{u_nxu_n^*}{y}| \leq \sum_{g\in\mathrm{Stab}(F)\cap \ell\mathrm{Stab}(F)k^{-1}}&|\tau((\otimes_{i\in F}y_i^*)u_n^g(\otimes_{i\in F}x_{i}))|\\
    &|\tau((\otimes_{i\in \ell F}y_i^*x_{g^{-1}i})\sigma_\ell(u_n^{\ell^{-1}gk})^*)|.
\end{align*}

We now use Cauchy-Schwarz to show get the following bound on $|\inp{u_nxu_n^*}{y}|$:

\begin{equation}
\label{ineq-CS}
\begin{split}
    |\inp{u_nxu_n^*}{y}| 
    \leq \Bigg(\sum_{g\in\mathrm{Stab}(F)\cap \ell\mathrm{Stab}(F)k^{-1}} & |\tau((\otimes_{i\in F}y_i^*)u_n^g(\otimes_{i\in F}x_{i}))|^2\Bigg)^{1/2}\\ 
    \Bigg(\sum_{g\in\mathrm{Stab}(F)\cap \ell\mathrm{Stab}(F)k^{-1}} & |\tau((\otimes_{i\in \ell F}y_i^*x_{g^{-1}i})\sigma_\ell(u_n^{\ell^{-1}gk})^*)|^2\Bigg)^{1/2} 
\end{split}
\end{equation}

We now observe that 
$$\|E_{L(\mathrm{Stab}(F))}((\otimes_{i\in F}y_i^*)u_n(\otimes_{i\in F}x_{i}))\|_2^2 = \sum_{g\in\mathrm{Stab}(F)}|\tau((\otimes_{i\in F}y_i^*)u_n^g(\otimes_{i\in F}x_{i}))|^2,$$ and the former goes to 0 by our choice of $u_n.$ A similar computation applies to the second factor in (\ref{ineq-CS}). Therefore $|\inp{u_nxu_n^*}{y}|\to0,$ proving Claim A.

Let $e\in \cB(L^2(M))$ be the projection from $L^2(M)$ onto $\cK = \overline{\mathrm{span}\{bu_g: b\in B^I, \ g\in L\}}^{\|\cdot\|_2}.$ Note that $e$ is a left $B^F\rtimes\mathrm{Stab}(F)$-modular map since $\mathrm{Stab}(F)L\subset L.$

\underline{\textbf{Claim B:}} For any nonzero projection $r\in M$, $e(rMr)\neq\{0\}$. 

Note that $e(rxr) = 0$ for all $x\in M$ if and only if $\inp{rxr}{\xi} = 0$ for all $\xi \in \cK$ if and only if $\inp{x}{r\xi r} =0$ for all $x,\xi$ as before. So it suffices to show that $r\xi r \neq 0$ for some $\xi\in \cK.$

We assume towards a contradiction that $r\cK r = 0.$
Note that elements of the form $u_gw$ where $g\in\Gamma$ and $w\in\cU(B^I)$ are a group that generate $M$ as a von Neumann algebra. Therefore the minimal 2-norm element of the closed convex hull of the $\{u_gwrw^*u_g^*:g\in\Gamma, w\in\cU(B^I)\}$ is $\tau(r)1$. Therefore we can take finitely many $t_i > 0$ that sum to 1, $g_i\in\Gamma$ and $w_i\in\cU(B^I)$ such that $\|\sum_{i=1}^nt_iu_{g_i}w_irw_i^*u_{g_i}^*-\tau(r)1\|_2 < \tau(r)^{3/2}$. Since $\Gamma\curvearrowright I$ has infinite orbits, there is $g\in \Gamma$ such that $g(\cup_{i=1}^ng_iF)\cap F = \emptyset.$ Hence $gg_i\in L$ for each $1\leq i\leq n.$ Then $\|\sum_{i=1}^nt_iru_gu_{g_i}w_irw_i^*u_{g_i}^*-ru_g\tau(r)\|_2 < \tau(r)^{3/2}$. But $u_{gg_i}w_i\in \cK$ so $ru_{gg_i}w_ir = 0$ for each $i.$ Hence $\tau(r)^{3/2} = \|ru_g\tau(r)\|_2  < \tau(r)^{3/2}$, a contradiction. This proves Claim B.

We now observe that $e(v^*Qv)\neq\{0\}$. Otherwise, we would have $e(v^*PQv) = e(v^*pPpQv) = e(\theta(pPp)v^*Qv) = \theta(pPp)e(v^*Qv) = \{0\}$. Taking a linear span and a 2-norm closure, this would imply $e(v^*L^2(M) v) = e(v^*vL^2(M) v^*v) = \{0\},$ contradicting Claim B. 

Take $y\in Q$ such that $e(v^*yv)\neq0.$ Claim A says that $\inp{u_ne(a)u_n^*}{b}\to 0$ for all $a,b\in L^2(M)$. Therefore 
\begin{align*}
    \|v^*yv\|_2^2 &= \inp{v^*yv}{v^*yv}\\
    &= \lim_{n\to\infty}\inp{u_n(v^*yv)u_n^*}{v^*yv} \\
    &=\lim_{n\to\infty}\inp{u_n(v^*yv-e(v^*yv))u_n^*}{v^*yv} + \lim_{n\to\infty}\inp{u_n(e(v^*yv))u_n^*}{v^*yv}\\
    &= \lim_{n\to\infty}\inp{u_n(v^*yv-e(v^*yv))u_n^*}{v^*yv}\\
    &\leq \liminf_{n\to\infty}|\inp{u_n(v^*yv-e(v^*yv))u_n^*}{v^*yv}|\\
    &\leq \|v^*yv-e(v^*yv)\|_2\|v^*yv\|_2\\
    &< \|v^*yv\|_2^2,
\end{align*}
which is a contradiction, completing the proof.
\end{proof}

Using similar techniques, we may show not only that $P$ embeds into a corner of some $L(\mathrm{Stab}(F))$ but also into a corner of \textit{every} $L(\mathrm{Stab}(F))$:

\begin{thm}
\label{thm-all-StabF}
Let $B$ be a nontrivial tracial von Neumann algebra and $\Gamma\curvearrowright I$ a free action with infinite orbits. Let $F\subset I$ be a finite subset. If $M = P\overline{\otimes}Q = B^I\rtimes \Gamma$ is a tensor decomposition such that $P\prec_M L\Gamma$, then $P\prec_M L(\mathrm{Stab}(F))$ for all $F\subset I$ finite.
\end{thm}
\begin{proof}
Fix F$\subset I$ finite. By Remark 3.8 of Vaes \cite{vaes2008explicit} there are projections $p\in P$ and $q\in L\Gamma$ and a *-homomorphism $\theta:pPp \to qL\Gamma q$ and a nonzero partial isometry $v\in pMq$ such that $xv=v\theta(x)$ for all $x\in pPp$, $vv^* \in pPp'\cap pMp = pQ, $ $v^*v \in \theta(pPp)'\cap qMq$ and if $A\subset L\Gamma$ is such that $P\not\prec_M A$ then $\theta(pPp) \not\prec_{L\Gamma} A.$

Suppose for a contradiction that $P\not\prec_M L(\mathrm{Stab}(F)).$ Then $P\not\prec_M L(\mathrm{Stab}(G))$ for all finite subsets $F\subset G\subset I.$ Hence there are unitaries $u_n^G\in \theta(pPp)$ such that $\|E_{L\mathrm{Stab}(G)}(xu_n^Gy)\|_2 \to 0$ for all $x,y\in L\Gamma.$ 

\underline{\textbf{Claim A:}} For all $x\in \overline{\mathrm{span}\{bu_g: b\in (B\ominus \bC)^G,\ g\in\Gamma\}}^{\|\cdot\|_2}$ and $y\in L^2(M),$ $$\inp{u_n^Gx(u_n^G)^*}{y}\to0.$$ 

To slightly ease the notation, we drop the superscript of $G$ on the $u_n$ until the end of the proof of this claim. It suffices to prove the claim on a $\|\cdot\|_2$-densely spanning set, so without loss of generality let $x = (\otimes_{i\in G}x_i)u_k$ where $\tau(x_i) = 0$ for all $i$ and let $y = (\otimes_{i\in H}y_i)u_\ell$ for some finite set $H\subset I$ and $k,\ell\in \Gamma.$ Set $x_i=y_j= 1$ for $i\not\in G$ and $j\not\in H.$ Write $u_n = \sum_{g\in\Gamma}u_n^gu_g.$ Then as in the previous theorem, we compute that 

$$\inp{u_nxu_n^*}{y} = \sum_{g\in\Gamma}u_n^g\overline{u_n}^{\ell^{-1}kg}\tau(\otimes_{i\in I}y_i^*x_{g^{-1}i}) $$

For $i\in G,$ the trace of $x_i$ is 0, so we see that $\tau(\otimes_{i\in I}y_i^*x_{g^{-1}i}) = \prod_{i \in gG}\tau(y_i^*x_{g^{-1}i})$ if $gG\subset H$ and is 0 otherwise. We therefore now know that 

$$\inp{u_nxu_n^*}{y} = \sum_{gG\subset H}u_n^g\overline{u_n}^{\ell^{-1}kg}\tau(\otimes_{i\in I}y_i^*x_{g^{-1}i}). $$

As in the previous theorem, we can bound $|\tau(\otimes_{i\in I}y_i^*x_{g^{-1}i})|$ by $\|x\|_2\|y\|_2.$ Therefore we know that $$|\inp{u_nxu_n^*}{y}| \leq \sum_{gG\subset H}|u_n^g\overline{u_n}^{\ell^{-1}kg}\tau(\otimes_{i\in I}y_i^*x_{g^{-1}i})| \leq \sum_{gG\subset H}|u_n^g\overline{u_n}^{\ell^{-1}kg}|\|x\|_2\|y\|_2.$$ Cauchy-Schwarz then implies this is at most $\sum_{gG\subset H}|u_n^g|^2\|x\|_2\|y\|_2.$

So it suffices to show that $\sum_{gG\subset H}|u_n^g|^2\to0.$ First, note that $\{g\in \Gamma: gG\subset H\}$ is equal to a finite union of left cosets of $\mathrm{Stab}(G),$ say $g_t\mathrm{Stab}(G)$ for $1\leq t\leq B.$ Then $\sum_{gG\subset H} |u_n^g|^2 = \sum_{t=1}^B\|E_{L\mathrm{Stab}(G)}(u_{g_t}^{-1}u_n)\|_2^2 \to 0.$ This proves Claim A.

Let $e_G\in \cB(L^2(M))$ be the projection from $L^2(M)$ onto $$\cK_G = \overline{\mathrm{span}\{bu_g: b\in (B\ominus \bC)^G,\ g\in\Gamma\}}^{\|\cdot\|_2}.$$ Note that $e$ is a right $L\Gamma$-modular map. 

\underline{\textbf{Claim B:}} Set $r = v^*v$, where $v$ is the partial isometry from the first paragraph of the proof of this theorem.  Then there exists $G\supset F$ finite such that $e_G(rMr)\neq\{0\}$.

Suppose for a contradiction the claim were false. Then $e_G(rMr) =\{0\}$ for all finite sets $G\subset I$ such that $G\supset F$. As in the proof of the previous theorem, this would imply that $r\cK_Gr = \{0\}$ for all such $G.$ Therefore $r\overline{\mathrm{span}\{\cK_G : G\supset F\}}^{\|\cdot\|_2} r = \{0\}.$ A straightforward computation shows that $$\overline{\mathrm{span}\{\cK_G : G\supset F\}}^{\|\cdot\|_2} = \overline{\mathrm{span}\{bu_g: b\in B^{I\setminus F} \overline{\otimes}(B\ominus \bC)^F,\ g\in\Gamma\}}^{\|\cdot\|_2}.$$ 

Now, pick an invertible element $v \in (B\ominus\bC)^F$. (Since $B$ is nontrivial, there is a projection $p\in B$ such that $t = \tau(p) \not\in\{0,1\}$, and we can take $v = p/t - (1-p)/(1-t).$ We remark that the element $v$ can be chosen to be unitary unless $B$ has a minimal central projection of trace greater than 1/2, as noted in Lemma A.1 of \cite{farah2014model}. 

Let $x$ be the element of minimal 2-norm in $$\overline{\mathrm{conv} \{u^*v^{-1}rvu : u\in \cU(B^{I\setminus F})\}}^{\|\cdot\|_2} .$$ Then $\tau(x) = \tau(r) \neq 0$ and $x\in (B^{I\setminus F})'\cap M \subset B^I$, with the last inclusion following from the fact the action of $\Gamma$ on $I$ is free. Now let $y$ be the element of minimal 2-norm in $$\overline{\mathrm{conv} \{u_g x u_g^* = \sigma_g(x) : g\in\Gamma\}}^{\|\cdot\|_2} .$$ Since $\sigma_g(x) \in B^I$ for all $g\in \Gamma,$ it is clear that $y\in B^I.$ But $y$ also commutes with $L\Gamma$ and so $y$ is a scalar, namely, $y = \tau(r)1.$ 

We now note that since $u_gu^*v^{-1}rvuu_g^* r = 0$ for each $g\in\Gamma$ and $u\in \cU(B^{I\setminus F}),$ we must have that $yr = 0$ too. But this implies that $r = 0,$ a contradiction. We have thus proved Claim B.

As in the proof of the previous theorem, we combine Claims A and B to contradict our assumption that $P \not\prec_M L(\mathrm{Stab}(F)).$ This completes the proof. \qedhere

\end{proof}

\begin{remark}
    To show that $(B^{I\setminus F})'\cap M \subset B^I$ as in the proof of Claim B of Theorem \ref{thm-all-StabF} above, it is sufficient to know either that (i) $B$ is nontrivial and $\Gamma\curvearrowright I$ is free or (ii) that $B$ is diffuse and $\Gamma\curvearrowright I$ is faithful. In either case, the action of $\Gamma$ on $B^I$ is properly outer by Lemmas \ref{lem-tensor-outer} and \ref{lem-tensor-outer-free}.
\end{remark}

We have now attained our goal for this section:

\begin{proof}[Proof of Theorem \ref{thm-case1}.]
    Combine Theorems \ref{thm-BF-StabF}, \ref{thm-exists-StabF}, and \ref{thm-all-StabF}.
\end{proof}

\section{Weak Bicentralization, Fullness, and Non-Rigid Tensor Factors}
\label{sect-nonrigid-factor}

In this section we will finish the proof of the main theorem. We will need to use the notion of weak bicentralization, first introduced in \cite{bannon2020full} in Lemma 3.4; the name was introduced in \cite{isono2019tensor} in Definition 4.1.

\begin{defn}
Let $M$ be a von Neumann algebra. We say that a subalgebra $N\subset M$ is \emph{weakly bicentralized} in $M$ if $_ML^2(\langle M,N\rangle)_M \prec _ML^2(M)_M.$
\end{defn}

\begin{prop}
\label{bicent-ultrapower}
Let $(M,\tau)$ be a tracial von Neumann algebra, and let $N\subset M$ be a subalgebra satisfying $N = (N'\cap M^\omega)'\cap M$ for some nonprincipal ultrafilter $\omega \in \beta\bN\setminus\bN.$ Then $N$ is weakly bicentralized in $M.$
\end{prop}

\begin{proof}
This is Lemma 3.4 in \cite{bannon2020full}. 
\end{proof}

We now give our first class of examples of weak bicentralization, which will be used later in Section \ref{sect-corollaries}.

\begin{prop}
\label{prop-wk-bicent-ez}
    Let $\Gamma\curvearrowright I$ be a group acting freely on a set, $B$ a nontrivial tracial von Neumann algebra, and let $M = B^I\rtimes\Gamma$ be the generalized Bernoulli crossed product. For a subset $G\subset I$, define $N_G = (Z(B)^{I\setminus G}\overline{\otimes} B^{G}).$ Let $F\subset I$ be finite. Then $N_F$ is weakly bicentralized in $M.$
\end{prop}

\begin{proof}
    To show that $N_F$ is weakly bicentralized in $M,$ it suffices to notice that $N_{I\setminus F} \subset N_F'\cap M$ and show that $N_{I\setminus F}'\cap M\subset N_F$. This is because then we would have $N_F'\cap M^\omega \supset N_F'\cap M \supset N_{I\setminus F}$ and so $(N_F'\cap M^\omega)'\cap M \subset N_{I\setminus F}'\cap M \subset N_F.$ (The inclusion $N_F\subset (N_F'\cap M^\omega)'\cap M$ is automatic.) We then apply Proposition \ref{bicent-ultrapower}.

    So, let $x=\sum_g x_gu_g\in N_{I\setminus F}'\cap M$. Then $x_gy=\sigma_g(y)x_g$ for all $y\in N_{I\setminus F}$. Let us first prove that $x \in N_F$. To do so, we need to show that $x_g = 0$ for all $g\neq e$ and that $x_e \in N_F$.

    If $g\neq e$, then we can argue as in Lemma \ref{lem-tensor-outer-free}: take a unitary $u\in B\ominus \bC$ and take $i_1,i_2,\ldots$ in $I\setminus F$ such that $gi_n\neq i_n$ for all $n\in\bN.$ Then set $y_n^{i_n} = u$ and $y^k_n = 1$ for $k\neq i_n$, and $y_n = \otimes_{i\in I} y_n^i.$ Then as in Lemma \ref{lem-tensor-outer-free}, $\inp{x_gy_n}{\sigma_g(y_n)x_g}\to 0$, which forces $x_g = 0$ since $x_gy_n=\sigma_g(y_n)x_g$.

    If $g=e$, then $x_ey=yx_e$ for all $y\in N_{I\setminus F}$ implies that $x_e \in N_{I\setminus F}'\cap B^I = N_F,$ finishing the proof.
\end{proof}

To prove our main theorem, our next goal will be to show that a different class of subalgebras of $M = B^I\rtimes\Gamma$ are weakly bicentralized in $M$. We introduce some notation to describe these subalgebras:

\begin{notation}
\label{partition}
Let $\Gamma\curvearrowright I$ be a group acting on a set. Let $F\subset I$ be a finite subset. Write $F = F_1\sqcup F_2\sqcup\ldots\sqcup F_k$ as a disjoint union. Define $$\mathrm{Norm}(\{F_i\colon 1\leq i\leq k\}) = \{g\in \Gamma\colon gF_i = F_i \text{ for all }1\leq i\leq k\}. $$
Note that $\mathrm{Stab}(F) = \mathrm{Norm}(\{\{i\}_{i\in F}\})$ and $\mathrm{Norm}(F) = \mathrm{Norm}(\{F\})$. We will write $\cF =\{F_i\colon 1\leq i\leq k\}$ for short, and call $\cF$ a \emph{partition} of $F.$

Likewise, if $B$ is a von Neumann algebra, define $B^\cF = B^{\{F_i\colon 1\leq i\leq k\}} := (B^F)^{\mathrm{Norm}(\cF)}$. In other words, $B^\cF$ consists of the elements of $B^F$ which are invariant under the action of $\mathrm{Norm}(\cF).$
\end{notation}

\begin{thm}
\label{cor-weak-bicent}
Let $\Gamma\curvearrowright I$ be a group acting faithfully on a set, $B$ a diffuse tracial von Neumann algebra, and let $M = B^I\rtimes\Gamma$ be the generalized Bernoulli crossed product. Let $F\subset I$ be finite and $\cF$ be a partition of $F$ as in Notation \ref{partition}. Define $N = (Z(B)^F\overline{\otimes} B^{I\setminus F})\rtimes \mathrm{Norm}(\cF).$ Then $N$ is weakly bicentralized in $M.$
\end{thm}

\begin{proof}

To show that $N$ is weakly bicentralized in $M,$ it suffices to notice that $B^\cF \subset N'\cap M$ and show that $(B^\cF)'\cap M\subset N$. This is because then we would have $N'\cap M^\omega \supset N'\cap M \supset B^\cF$ and so $(N'\cap M^\omega)'\cap M \subset (B^\cF)'\cap M \subset N.$ (The inclusion $N\subset (N'\cap M^\omega)'\cap M$ is automatic.) We then apply Proposition \ref{bicent-ultrapower}.

So, let $x=\sum_g x_gu_g\in (B^{\cF})'\cap M$. Then $yx_g=x_g\sigma_g(y)$ for all $y\in B^\cF$. Our goal is to show that $x\in N$. In other words, we wish to show that $x_g \in Z(B)^F\overline{\otimes} B^{I\setminus F}$ for all $g\in \mathrm{Norm}(\cF)$ and $x_g=0$ for $g\not\in \mathrm{Norm}(\cF).$

If $g\in \mathrm{Norm}(\cF)$ then $x_gy=yx_g$ for all $y\in B^\cF$. By Lemma \ref{lem-outer-commutant} we conclude that $x_g \in (B^\cF)' \cap B^I = (B^{\cF})'\cap (B^F \overline{\otimes} B^{I\setminus F})= Z(B)^F \overline{\otimes} B^{I\setminus F}$.

If $g\not\in \mathrm{Norm}(\cF)$ then this forces $x_g = 0$ by a similar argument to Lemma \ref{lem-tensor-outer}: take unitaries $u_n\to 0$ weakly in $B$, and pick $i\in F_m\in \cF$ such that $gi = j\not\in F_m$. Then set $y_n = \bigotimes_{k\in F_m}u_n \otimes \bigotimes_{k\not\in F_m}1 \in B^\cF.$ Then for all $x\in B^I,$ $\tau(xy_nx^*\sigma_g(y_n)^*)\to0$ as in Lemma \ref{lem-tensor-outer} and so it cannot be that $xy_n=\sigma_g(y_n)x$ for all $n$ and $x\neq0.$ 

Hence $N$ is weakly bicentralized in $M.$ \qedhere

\end{proof}

\begin{remark}
There are 2 main places the argument in the previous proposition breaks down if $B$ is not assumed to be diffuse. First, $\mathrm{Norm}(\cF)$ would not act properly outerly on $B^F$, and so $(B^\cF)'\cap B^F$ would not necessarily equal $Z(B)^F.$ However, this is only a minor issue since $(B^\cF)'\cap B^F$ is finite dimensional. 

More seriously, $B^\cF$ would be finite dimensional if $B$ were, and so $(B^\cF)'\cap M$ would be finite index in $M.$ In particular, $(B^\cF)'\cap M$ would not be contained in $N$, an infinite index subalgebra of $M.$ 
\end{remark}

We can now prove the following theorem: 

\begin{thm}
\label{thm-case2}
Let $\Gamma\curvearrowright I$ be a faithful, nonamenable action with the VV property of a countable group on a countable set. Let $(B,\tau)$ be a diffuse tracial von Neumann algebra. Write $M = B^I\rtimes\Gamma$ and suppose $M = P\overline{\otimes} Q$ is a tensor decomposition. Set $\Tilde{B} = B * L\bZ$ and $\Tilde{M} = \Tilde{B}^I\rtimes\Gamma$. If $Q'\cap \Tilde{M}^\omega \not\subset M^\omega,$ then there exists a finite subset $F\subset I$ such that $Q\prec_M (Z(B)^F\overline{\otimes}B^{I\setminus F})\rtimes \mathrm{Stab}(F)$.
\end{thm}

\begin{proof}

As $M$-$M$-bimodules, we have that $L^2(\Tilde{M})\ominus L^2(M) = \bigoplus_{A\in\cC} L^2(\langle M,A\rangle)$ where $\cC$ is a certain class of subalgebras of $M.$ By the proof of Lemma 5 in \cite{chifan2010ergodic}, $\cC$ consists of algebras of the form $A_\cF = B^{I\setminus F} \rtimes \mathrm{Norm}(\cF)$ where $F\subset I$ is finite and $\cF$ is a partition of $F$ as in Notation \ref{partition}.

By Lemma \ref{lem-commutant-is-commuting-square}, $Q'\cap \Tilde{M}^\omega$ and $M^\omega$ form a commuting square. Since $Q'\cap \Tilde{M}^\omega \not\subset M^\omega,$ Lemma \ref{lem-comm-sq-orthog} says there is $0\neq x = (x_n)_n \in Q'\cap \Tilde{M}^\omega$ such that $E_{M^\omega}(x) = 0.$ Scaling and perturbing the $x_n$ as needed, we obtain $x_n\in \Tilde{M}$ such that $\|x_n\|_2=1,$ $E_M(x_n) = 0,$ and $\lim_{n\to\omega} \|x_ny-yx_n\|_2 = 0$ for all $y\in Q.$ The $x_n$ are now almost-central vectors in $L^2(\Tilde{M})\ominus L^2(M)$, showing that $L^2(Q)\prec L^2(\Tilde{M})\ominus L^2(M)$.

By Theorem \ref{cor-weak-bicent}, for $A_\cF\in\cC$ we have that $N_\cF = A_\cF \overline{\otimes} Z(B)^F$ is weakly bicentralized in $M.$ I.e., $L^2(\langle M,N_\cF\rangle) \prec L^2(M)$ as $M$-$M$-bimodules.

Therefore $L^2(\langle M,N_\cF\rangle) \prec L^2(Q)\otimes L^2(P) \prec L^2(Q)$ as $Q$-$Q$-bimodules. 

By Theorem \ref{thm-weak-contain-bmo}, since $L^2(Q) \prec \bigoplus_{A_\cF\in\cC}L^2(\langle M,A_\cF\rangle)$ there is a conditional expectation $E:\bigoplus_{A_\cF\in\cC}\langle M,A_\cF\rangle \to Q.$ Here, $\bigoplus_{A_\cF\in\cC}\langle M,A_\cF\rangle$ is the direct sum in the category of von Neumann algebras, i.e., the $\ell^\infty$ direct sum, and $Q$ is viewed as a subalgebra via the diagonal embedding. We now observe that $E$ restricts to an expectation from $\bigoplus_{\cC} \langle M,N_\cF\rangle \to Q$ so that $L^2(Q) \prec \bigoplus_{\cC}L^2(\langle M,N_\cF\rangle) \prec \bigoplus L^2(M) \prec L^2(Q)$ as $Q$-$Q$-bimodules.

Since $\Gamma\curvearrowright I$ is faithful, nonamenable and has the VV property, and $B$ is diffuse, we have that $M$ is a full factor. Therefore so is $Q$. By Proposition 3.2 of \cite{bannon2020full}, since $Q$ if a full factor and $L^2(Q) \prec \bigoplus_{\cC}L^2(\langle M,N_\cF\rangle) \prec L^2(Q)$ as $Q$-$Q$-bimodules, we deduce that $L^2(Q) \subset \bigoplus_{\cC} L^2(\langle M,N_\cF\rangle)$.

Then for some partition $\cF$, $L^2(\langle M,N_\cF\rangle)$ has a nonzero $Q$-central vector. By Theorem \ref{thm-popa-fundamental}, $Q\prec_M N_\cF.$ In other words, $Q\prec_M (Z(B)^F \overline{\otimes} B^{I\setminus F})\rtimes \mathrm{Stab}(F)$. \qedhere

\end{proof}

We have now proved the main theorem of our paper:

\begin{proof}[P\textbf{roof of Theorem \ref{thm-main}.}]
    Combine Theorem \ref{thm-case1} and Theorem \ref{thm-case2}.
\end{proof}

\section{Applications}
\label{sect-corollaries}
\begin{cor}
\label{main-cor2}
Let $\Gamma\curvearrowright I$ be a faithful, nonamenable, transitive action with the VV property of a countable group on a countable set such that $\mathrm{Stab}(F)$ is amenable for some finite set $F\subset I$. Let $(B,\tau)$ be a \twoone factor. Write $M = B^I\rtimes\Gamma$. Then $M$ is prime.
\end{cor}

\begin{proof}
Write $M = P\overline{\otimes} Q$ so that $P$ and $Q$ are full factors. It is not possible for $P\prec_M L(\mathrm{Stab}(G))$ for all $G\subset I$ finite since $\mathrm{Stab}(F)$ is amenable for some finite set $F\subset I$. Similarly, it is not possible for $Q\prec_M L(\mathrm{Stab}(G))$ for all $G\subset I$ finite. Theorem \ref{thm-main} therefore tells us that there exist nonempty finite sets $F,G\subset I$ such that $P\prec_M B^{I\setminus F}\rtimes \mathrm{Stab}(F)$ and $Q\prec_M B^{I\setminus G}\rtimes \mathrm{Stab}(G).$ 

Since the action of $\Gamma$ on $I$ is transitive, there is $g\in \Gamma$ such that $\exists i\in  gG\cap F.$ Since $B^{I\setminus F}\rtimes \mathrm{Stab}(F) \subset B^{I\setminus \{i\}}\rtimes \mathrm{Stab}(i) $, we have that $P\prec_M B^{I\setminus \{i\}}\rtimes \mathrm{Stab}(i).$ Similarly, $Q\prec_M B^{I\setminus \{g^{-1}i\}}\rtimes \mathrm{Stab}(g^{-1}i)$. Now note that $u_g(B^{I\setminus \{g^{-1}i\}}\rtimes \mathrm{Stab}(g^{-1}i))u_g^{-1} = B^{I\setminus \{i\}}\rtimes \mathrm{Stab}(i)$. It now follows that $Q\prec_MB^{I\setminus \{i\}}\rtimes \mathrm{Stab}(i)$.


Now, taking relative commutants and applying Lemma 3.5 of \cite{vaes2008explicit}, we see that $B^i\prec_M P,Q$. By Lemma 2.4(3) of \cite{drimbe2019prime}, there are nonzero projections $$p,q\in Z((B^i)'\cap M) = Z(B^{I\setminus \{i\}}\rtimes \mathrm{Stab}(i)) = \bC1$$ such that $pB^i \prec^s_M P$ and $qB^i \prec^s_M Q$. But these projections are both 1, so $B^i\prec_M^s P,Q.$ By Lemma 2.8 of \cite{drimbe2019prime}, we have that $B^i \prec_M P\cap Q = \bC1$ which is impossible since $B$ is diffuse. This contradicts our assumption that $M = P\overline{\otimes} Q$ for \twoone factors $P,Q,$ so we conclude that $M$ is prime.
\end{proof}

\begin{proof}[P\textbf{roof of Theorem \ref{cor-main}.}] Since $\Gamma$ is not inner amenable, $\Gamma\curvearrowright I$ has the VV property by Proposition \ref{prop-vv-ez}. Apply Corollary \ref{main-cor2}. \qedhere
    
\end{proof}

    

We can also recover and generalize Theorem C of Isono-Marrakchi \cite{isono2019tensor} by noting that the left multiplication of $\Gamma$ on itself is free, ergodic, has the VV property, and has amenable stabilizers. We first state a lemma which is proved as part of Lemma 7.2 in \cite{isono2019tensor}:

\begin{lem}
\label{isono-marrakchi-7.2}
    Let $\Gamma\curvearrowright I$ be a free action with infinite orbits. Let $(B,\tau)$ be a nontrivial tracial von Neumann algebra. Write $M = B^I\rtimes\Gamma$. Let $C\subset B$ be a possibly trivial von Neumann subalgebra. Let $Q\subset M$ be a von Neumann subalgebra such that $Q \prec_M (C^I\vee B^F)$ for some finite set $F\subset I$ but $Q\not\prec_M C^I$. Then $\cN_{pMp}(pQ)''\prec_M B^I\rtimes \mathrm{Stab}(F)$ for some projection $p \in Q'\cap M.$
\end{lem}

\begin{proof}
    The same as the proof of Lemma 7.2 in \cite{isono2019tensor}, except for the last line.
\end{proof}

\begin{proof}[P\textbf{roof of Theorem \ref{cor-main-2}.}]
    Write $M = P\overline{\otimes}Q.$ Note that we cannot apply Theorem \ref{thm-main} directly since $B$ is not assumed here to be diffuse. However, the proof of Theorem \ref{thm-case1} does not require diffuseness, so its conclusion still applies. But it is not possible for $P\prec_M L(\mathrm{Stab}(F))$ as $\mathrm{Stab}(F)$ is amenable. So it must be that we are in the setting of Theorem \ref{thm-case2}. Namely, it must be that $Q'\cap \Tilde{M}^\omega\not\subset M^\omega$.

    That is to say, it must be that $L^2(Q) \prec \bigoplus_F L^2\left(\inp{M}{B^{I\setminus F}\rtimes\mathrm{Stab}(F)}\right).$ By Theorem \ref{thm-weak-contain-bmo}, we obtain a conditional expectation $E$ from $ \bigoplus_F \inp{M}{B^{I\setminus F}\rtimes\mathrm{Stab}(F)}$ down to the diagonal embedding of $Q.$ This restricts to a conditional expectation of $\bigoplus_F \inp{M}{B^{I}\rtimes\mathrm{Stab}(F)}$ onto $Q$. We therefore have that $$L^2(Q) \prec \bigoplus_F L^2\left(\inp{M}{B^{I}\rtimes\mathrm{Stab}(F)}\right).$$ By Lemma \ref{lem-rel-amen-bimods} we now deduce that $L^2(Q)\prec \bigoplus_F L^2\left(\inp{M}{B^{I}}\right).$ Since weak containment doesn't see infinite direct sums of the same bimodule, we have that $ L^2(Q)\prec L^2(\inp{M}{B^{I}}).$

     Since $B^I$ is the increasing union of $Z(B)^{I\setminus F}\overline{\otimes} B^F,$ Proposition 2.5 of \cite{isono2019tensor} says that $$L^2(Q) \prec \bigoplus_F L^2\left(\inp{M}{Z(B)^{I\setminus F}\overline{\otimes} B^F}\right).$$ By Proposition \ref{prop-wk-bicent-ez} the algebras $Z(B)^{I\setminus F}\overline{\otimes} B^F$ are all weakly bicentralized and so we have
    $$L^2(Q) \prec \bigoplus_F L^2\left(\inp{M}{Z(B)^{I\setminus F}\overline{\otimes} B^F}\right) \prec \bigoplus L^2(M) \prec L^2(Q).$$ By Proposition 3.2 of \cite{bannon2020full}, we have that since $Q$ is full, $$L^2(Q) \subset \bigoplus_F L^2(\inp{M}{Z(B)^{I\setminus F}\overline{\otimes} B^F}).$$ Therefore $L^2(Q) \subset L^2(\inp{M}{Z(B)^{I\setminus F}\overline{\otimes} B^F})$ for some finite set $F\subset I$. Hence $Q \prec_M Z(B)^{I\setminus F}\overline{\otimes} B^F$ for some $F\subset I$. We now apply Lemma \ref{isono-marrakchi-7.2} with $C = Z(B)$, which says that either $Q\prec_M Z(B)^I$ or $pMp \prec_M B^I\rtimes \mathrm{Stab}(F)$ for some projection $p\in P = Q'\cap M$. The latter is impossible since $\mathrm{Stab}(F)$ has infinite index in $\Gamma,$ and the former is impossible since we are assuming $Q$ is a diffuse tensor factor of $M,$ and is therefore full, and hence nonamenable. We conclude that $M$ must be prime.
\end{proof}

We finish with some remarks about when the conclusion of Theorem \ref{thm-case1} cannot hold. We repeat Ioana and Spaas's remark from \cite{ioana2021ii1} that condition (2) in the following proposition holds when the subgroups $\mathrm{Stab}(F)$ are the tails of either a direct sum or of a free product.

\begin{prop}
Let $\Gamma\curvearrowright I$ be a faithful, nonamenable action with the VV property of a countable group on a countable set. Let $(B,\tau)$ be a nontrivial tracial von Neumann algebra. Write $M = B^I\rtimes\Gamma$ and suppose $M = P\overline{\otimes} Q$ is a tensor decomposition of $M$ where $P$ and $Q$ are both \twoone factors. Suppose that any of the following conditions are satisfied:
\begin{enumerate}
    \item[(a)] There is $F\subset I$ finite such that $\mathrm{Stab}(F)$ is amenable.
    \item[(b)] There are von Neumann subalgebras $M_F \subset M$ such that \newline $L^2(M) \prec L^2(\mathrm{Stab}(F)) \overline{\otimes} L^2(M_F)$ as $L(\mathrm{Stab}(F))$-$M_F$ bimodules for all $F\subset I$ finite and $\|x- E_{M_F}(x)\|_2\to 0$ for all $x\in M$.
    \item[(c)] $B = L^\infty(X)$ is a diffuse abelian von Neumann algebra. Let $\cR$ be the equivalence relation for the generalized Bernoulli action of $\Gamma$ on $X^I$. Assume that $\cR = \cR_1\times \cR_2$ is a product decomposition of $\cR$ and that $P = L(\cR_1)$ and $Q = L(\cR_2).$ 
\end{enumerate}
Then there is $F\subset I$ finite such that $P\not\prec_M L(\mathrm{Stab}(F)).$
\end{prop}

\begin{proof}
If (a) holds, then the conclusion follows since $M$ is full and hence $P$ is full, thus nonamenable.

If (b) holds and $P\prec_M L(\mathrm{Stab}(F))$ for all $F\subset I$ finite, then in particular $P\lessdot_M L(\mathrm{Stab}(F))$, and by Lemma 2.6 of \cite{ioana2021ii1}, we have that $P \lessdot_M \cap_{F}L(\mathrm{Stab}(F)) = \bC1.$ So then $P$ is amenable, again contradicting the fullness of $P.$


If (c) holds, note that since $\mathcal{R} = \mathcal{R}_1\times\mathcal{R}_2$ we have that $L^\infty(X_i)\subset L^\infty(X)$ for $i=1,2.$ Note too that $X_1$ is diffuse since $L^\infty(X_1)$ is a MASA in the \twoone factor $P.$ Furthermore, we know that $L^\infty(X_1) \subset P$ embeds in a corner of $L\Gamma$ if $P$ does. But note that if $v_n\to 0$ in $L^\infty(X_1)$ and for $b,c\in L^\infty(X)$ and $g,h\in \Gamma,$ $E_{L\Gamma}(bu_gv_ncu_h) = E_{L\Gamma}(b\sigma_g(v_nc)u_{gh}) = \tau(b\sigma_g(v_nc)) = \tau(\sigma_{g^{-1}}(b)v_nc) \to 0.$ So it cannot be that $L^\infty(X_1)\prec_M L\Gamma.$ Hence $P \not\prec_M L(\mathrm{Stab}(F))$ for any $F$. \qedhere

\end{proof}

In the setting of equivalence relations, we can therefore get slightly further than Theorem \ref{thm-main}:

\begin{cor}
    Let $\Gamma\curvearrowright I$ be a faithful, nonamenable action with the VV property of a countable group on a countable set. Let $X$ be a nonatomic probability space. Then $\Gamma\curvearrowright X^I$ is a free, nonamenable, ergodic action. Let $\cR = \cR(\Gamma\curvearrowright X^I)$ be the corresponding orbit equivalence relation. Let $\cR = \cR_1\times \cR_2$ be a product decomposition. Write $M = L(\cR) = L^\infty(X^I)\rtimes\Gamma$. Then there exist nonempty finite sets $F,G\subset I$ such that $L(\cR_1) \prec_M L^\infty(X^I)\rtimes \mathrm{Stab}(F)$ and $L(\cR_2) \prec_M L^\infty(X^I)\rtimes\mathrm{Stab}(G).$
\end{cor}

\end{document}